\documentclass[11pt]{article}

\title{\textbf{Lagrangian torus invariants using $ECH=SWF$}}

\author{\Large Chris Gerig}

\AtEndDocument{\bigskip{\footnotesize
  \textsc{Department of Mathematics, Harvard University, MA 02138, USA} \par  
  \textit{E-mail address:} \texttt{cgerig@math.harvard.edu}
}}

\usepackage{hyperref}
\usepackage[initials]{amsrefs}

\usepackage{fullpage}
\usepackage{amssymb,latexsym,amsmath,amsthm,amscd,dsfont,graphicx}
\usepackage[all]{xy}
\usepackage{comment,indentfirst,enumerate}
\usepackage[usenames,dvipsnames,svgnames,table]{xcolor}

\numberwithin{equation}{section}

\newtheorem{theorem}{Theorem}[section]
\newtheorem{prop}[theorem]{Proposition}
\newtheorem{cor}[theorem]{Corollary}
\newtheorem{lemma}[theorem]{Lemma}
\newtheorem{lemma-definition}[theorem]{Lemma-Definition}

\theoremstyle{definition}
\newtheorem{definition}[theorem]{Definition}
\newtheorem{remark}[theorem]{Remark}

\newtheorem{convention}[theorem]{Convention}

\renewcommand{\SS}{{\mathbb S}}

\newcommand{\CC}{{\mathbb C}}
\newcommand{\QQ}{{\mathbb Q}}

\newcommand{\RR}{{\mathbb R}}

\newcommand{\NN}{{\mathbb N}}

\newcommand{\ZZ}{{\mathbb Z}}
\newcommand{\TT}{{\mathbb T}}

\newcommand{\nN}{{\mathcal N}}
\newcommand{\aA}{{\mathcal A}}
\newcommand{\cC}{{\mathcal C}}

\newcommand{\mM}{{\mathcal M}}
\newcommand{\vV}{{\mathcal V}}

\newcommand{\bB}{{\mathcal B}}

\newcommand{\lL}{{\mathcal L}}

\newcommand{\op}{\operatorname}
\newcommand{\Diff}{\op{Diff}}

\newcommand{\Spinc}{\op{Spin}^c}

\newcommand{\End}{\op{End}}

\renewcommand{\ker}{\op{Ker}}

\newcommand{\cl}{\op{cl}}
\newcommand{\PD}{\op{PD}}
\newcommand{\gr}{\op{gr}}

\renewcommand{\1}{\mathds{1}}
\newcommand{\e}{\varepsilon}

\newcommand{\crit}{\op{crit}}

\newcommand{\grad}{\op{grad}}

\newcommand{\bA}{\mathbf A}

\newcommand{\fs}{\mathfrak s}

\newcommand{\fp}{\mathfrak p}
\newcommand{\fq}{\mathfrak q}

\newcommand{\fd}{\mathfrak d}
\newcommand{\fM}{\mathfrak M}

\newcommand{\fc}{\mathfrak c}

\newcommand{\Hfrom}{\widehat{\mathit{HM}}}
\newcommand{\Cfrom}{\widehat{\mathit{CM}}}

    \RequirePackage{rotating}                   % Case (2)
    \def\Hto{%
       \setbox0=\hbox{$\widehat{\mathit{HM}}$}
       \setbox1=\hbox{$\mathit{HM}$}
       \dimen0=1.1\ht0
       \advance\dimen0 by 1.17\ht1
       \smash{\mskip2mu\raise\dimen0\rlap{%
          \begin{turn}{180}
              {$\widehat{\phantom{\mathit{HM}}}$}
           \end{turn}} \mskip-2mu    
                \mathit{HM}
    }{\vphantom{\widehat{\mathit{HM}}}}{}}
    \RequirePackage{rotating}                   % Case (2)
    \def\Cto{%
       \setbox0=\hbox{$\widehat{\mathit{CM}}$}
       \setbox1=\hbox{$\mathit{CM}$}
       \dimen0=1.1\ht0
       \advance\dimen0 by 1.17\ht1
       \smash{\mskip2mu\raise\dimen0\rlap{%
          \begin{turn}{180}
              {$\widehat{\phantom{\mathit{CM}}}$}
           \end{turn}} \mskip-2mu    
                \mathit{CM}
    }{\vphantom{\widehat{\mathit{CM}}}}{}}
    
\numberwithin{equation}{section}

\definecolor{blue}{rgb}{0,0,1}
\definecolor{red}{rgb}{1,0,0}
\definecolor{green}{rgb}{0,.7,0}

\begin{document}
\maketitle

\begin{abstract}
We construct distinguished elements in the embedded contact homology (and monopole Floer homology) of a 3-torus, associated with Lagrangian tori in symplectic 4-manifolds and their isotopy classes. They turn out not to be new invariants, instead they repackage the Gromov (and Seiberg--Witten) invariants of various torus surgeries. We then recover a result of Morgan--Mrowka--Szab\'o on product formulas for the Seiberg--Witten invariants along 3-tori.
\end{abstract}

%%%%%%%%%%%%%%%%%%%%%%%%%%%%%%%%%%%%%%%%%%%%%%%%%%%%
%%%%%%%%%%%%%%%%%%%%%%%%%%%%%%%%%%%%%%%%%%%%%%%%%%%%
\section{Introduction and main result}
\label{Introduction and main result}

To study Lagrangian submanifolds of a given symplectic 4-manifold, one mantra is to replace ``Lagrangian boundary conditions'' with ``asymptotic Reeb orbit conditions'' for the development of tools involving pseudoholomorphic curves. As a clarification, fix a tubular neighborhood of the Lagrangian. Then instead of considering tools defined by counts of pseudoholomorphic curves in the 4-manifold that have boundary on the Lagrangian, we count pseudoholomorphic curves in the complement of the tubular neighborhood that have punctures asymptotic to certain circles along the boundary of the tubular neighborhood. These tools are inherently related, by shrinking the tubular neighborhood to its core. One such tool is the \textit{ECH cobordism map} induced by the complement of a suitable tubular neighborhood of a Lagrangian torus. Can ECH then help detect knottedness of Lagrangian tori? The general problem of finding non-isotopic Lagrangian tori has been well studied before using other techniques, see for example \cite{DGI, EP:Luttinger, EP:problem, EP:unknotted, CieliebakMohnke:punctured, Vidussi:LagrSurfaces, FS:Lagrangian, Knapp:LagrTori, Mohnke:needle}.

Fix a Lagrangian torus $L\subset (X,\omega)$ in a connected symplectic 4-manifold. By the (Weinstein) Lagrangian neighborhood theorem, the unit disk cotangent bundle of $L$ symplectically embeds in $X$ as a tubular neighborhood of $L$ (sending the 0-section to $L$) such that $UT^* L$ is a contact hypersurface equipped with the tautological 1-form $\lambda_0$. In particular, since $\nN$ may also be viewed as the unit disk normal bundle of $L$, this torus must have self-intersection number zero:
$$L\cdot L=c_1(\text{normal bundle})=c_1(T^* L)=-\chi(L)=0$$
Taking such an arbitrarily small Weinstein tubular neighborhood $\nN$, the canonical trivialization $T^* L\cong L\times\CC^2$ in terms of conjugate variables defines a framing of $\nN$. The resulting symplectic manifold
$$(X_0,\omega):=(X-\nN,\omega|_{X-\nN})$$
has concave tight contact boundary ($UT^* L,\lambda_0)$.

Regarding the relation between Lagrangian boundary conditions and asymptotic Reeb orbit conditions, the Reeb vector field on $(UT^* L,\lambda_0)$ generates the geodesic flow of the flat product Riemannian metric on $L$. That is, we can identify each point of $T^* L$ with a point of $TL$ using the metric, and then the value of the Reeb field at that point of the (unit) cotangent bundle is the corresponding (unit) vector parallel to $L$. So the Reeb orbits correspond to the geodesics under the projection $UT^* L\to L$. Regarding the relation between pseudoholomorphic curves in either $X$ or $X_0$, such surfaces represent relative homology classes in either $H_2(X,L;\ZZ)$ or $H_2(X_0,UT^* L;\ZZ)$ and these classes are identified by excision: $H_2(X_0,UT^* L;\ZZ)\cong H_2(X,\nN;\ZZ)\cong H_2(X,L;\ZZ)$.

Parametrize the 2-torus
$$L\cong T^2=S^1\times S^1$$
by $(x,y)\in(\RR/2\pi\ZZ)^2$ and parametrize the 3-torus total space
$$UT^* L\cong T^3=S^1\times S^1\times S^1$$
by $(x,y,\theta)\in(\RR/2\pi\ZZ)^3$, so that
$$\lambda_0=\cos\theta\,dx+\sin\theta\,dy$$
Its contact structure is $\xi_0=\langle\partial_\theta,\sin\theta\,\partial_x-\cos\theta\,\partial_y\rangle$, and $c_1(\xi_0)=0$ because the section $\partial_\theta\in\Gamma(\xi_0)$ has no zeros.

Associated with the contact manifold $(T^3,\xi_0)$ is its embedded contact homology $ECH_*(T^3,\xi_0)$, a Floer homology whose chain complex is (roughly speaking) generated by sets of Reeb orbits with respect to a generic rescaling of $\lambda_0$. To define invariants of $L$, we first count certain pseudoholomorphic curves in a completion of $(X_0,\omega)$ which assemble into a cycle in the ECH chain complex of $T^3$ for a particular rescaling of $\lambda_0$. Then we use a gluing formula (decomposing $X=X_0\cup\nN$) to compute this element in ECH explicitly. Specifically,

\bigskip\noindent
\textbf{Main Result 1.} \textit{Given $(X,\omega,L)$ as above such that $X$ is minimal\footnote{A manifold is \textit{minimal} if there are no \textit{exceptional spheres}, smoothly embedded 2-spheres of self-intersection $-1$. Such spheres have negative (ECH) index multiple covers and cause complications when trying to count them.}, fix $A\in H_2(X_0,\partial X_0;\ZZ)$ such that $\partial A=0$. There is a deformation of $\nN$ such that $\partial X_0$ carries a nondegenerate contact form $\lambda_A$ with contact structure $\xi_0$ (see Lemma~\ref{lem:nbhd}), so that the following holds. For suitably generic almost complex structure $J$ on a completion of $(X_0,\omega)$ there is a well-defined element
$$Gr_L(A)\in ECH_0(T^3,\xi_0)\cong (\ZZ/2)^3$$
given by weighted counts of (possibly disjoint and multiply covered) $J$-holomorphic curves in the completion of $(X_0,\omega)$ which represent $A$ and are asymptotic to Reeb orbits with respect to $\lambda_A$ whose total homology class in $T^3$ is trivial (see Theorem~\ref{thm:class} and Definition~\ref{defn:Gr}).}

\textit{Identify this Floer homology with $H_2(T^3;\ZZ/2)\cong (\ZZ/2)^3$ whose basis is determined (as specified in Section~\ref{Floer homologies of 3-tori}) by a fixed basis $\lbrace x,y,\bar\theta\rbrace$ of $H_1(T^3;\ZZ/2)$. Then
$$Gr_L(A)=\left(Gr_{X_{f_{1,1,0}}}(A)-Gr_X(A),\;Gr_{X_{f_{1,0,1}}}(A)-Gr_X(A),\;Gr_X(A)\right)$$
where $X_{f_{1,r,s}}$ is the symplectic manifold obtained by Luttinger surgery of $X$ along $L$ using the contactomorphism\footnote{This map is given explicitly by $f_{1,r,s}(x,y,\theta)=(x+r\theta,y+s\theta,\theta)$.} $f_{1,r,s}:(T^3,\xi_0)\to(T^3,\xi_0)$ determined by $f_*\bar\theta=\bar\theta-rx-sy$ for $r,s\in\ZZ$, and $Gr_{X_{f_{1,r,s}}}(A)$ denotes the sum of Taubes' Gromov invariants (modulo 2) of $X_{f_{1,r,s}}$ with respect to all lifts $\tilde A\in H_2(X_{f_{1,r,s}};\ZZ)$ of $A$ (see Theorem~\ref{thm:main}).}

\begin{remark}
We can define $Gr_L(A)$ over $\ZZ$, by equipping our ECH generators and moduli spaces of pseudoholomorphic curves with ``coherent orientations'' (see \cite{Gerig:taming}). Although we work over $\ZZ/2$ when relating $Gr_L(A)$ to Seiberg--Witten theory, we expect the arguments to work over $\ZZ$ (see \cite{Gerig:Gromov}).
\end{remark}

\begin{remark}[\textit{Relation to Gromov--Witten theory}]
Auroux brought up the following paradox. Consider $L=S^1\times S^1\subset\CC^2\subset\CC P^2$, viewed as a Lagrangian torus in $(\CC P^2,\omega_\text{std})$, and a certain moduli space of $J$-disks representing a fixed homotopy class in $\pi_2(\CC P^2,L)$. There is a wall-crossing phenomenon in which the $J$-disks disappear (see \cite{Auroux:mirrorTduality,Vianna:exoticTori}). In the ECH picture after neck-stretching along a neighborhood of the torus, presumably these disks contribute to the ECH invariants $Gr_L(A)$. But $Gr_L(A)$ should not depend on $J$, so what is going on? Perhaps it is the case that when the disks die there are other curves of higher genus which birth to compensate the discrepancy, representing the relevant class in $H_2(\CC P^2,L;\ZZ)$? It turns out that a result from \cite{Tonkonog:string} implies that wall-crossing does not occur for the collection of $J$-disks whose corresponding collection of boundary loops on the torus sum to zero on homology (strictly speaking, all of the wall-crossings cancel out).
\end{remark}

Along the way to computing our ECH invariant $Gr_L(A)$, we pass through Seiberg--Witten theory. Namely, there is an isomorphism $ECH_0(T^3,\xi_0)\cong\Hto(T^3,\fs_0)$ for a version of monopole Floer homology with respect to the unique torsion spin-c structure $\fs_0$ on $T^3$ (see Section~\ref{Floer homologies of 3-tori}), and $Gr_L(A)$ may be expressed as a suitable count of Seiberg--Witten solutions on the completion of $X_0$. Then we use a gluing formula in monopole Floer homology to compute those Seiberg--Witten counts. As this part does not care about the symplectic structure, we recover the following main result of \cite{MMS:product}.

\bigskip\noindent
\textbf{Main Result 2.} \textit{Let $X$ be a connected smooth 4-manifold with $b^2_+(X)>0$, let $L$ be a smoothly embedded 2-torus with trivial self-intersection number, and let $\fs$ be a spin-c structure on the complement $X_0$ of a tubular neighborhood of $L$ such that $\fs|_{\partial X_0}=\fs_0$. Fix a basis $\lbrace x,y,\bar\theta\rbrace$ of $H_1(T^3;\ZZ)\cong\ZZ^3$ (as specified in Section~\ref{Floer homologies of 3-tori}) and denote by $f_{p,r,s}:T^3\to T^3$ an orientation-preserving diffeomorphism determined by $f_*\bar\theta=p\bar\theta-rx-sy$ for $p,r,s\in\ZZ$. Then there is a product formula along $T^3$ for Seiberg--Witten invariants
$$SW_{X_{f_{p,r,s}}}(\fs)=p\cdot SW_X(\fs)+r\cdot SW_{X_{f_{0,1,0}}}(\fs)+s\cdot SW_{X_{f_{0,0,1}}}(\fs)$$
where in the $b^2_+(X)=1$ case the Seiberg--Witten invariants are computed in corresponding chambers (see Corollary~\ref{cor:main}). Here, $X_{f_{p,r,s}}$ is the logarithmic transformation of $X$ along $L$ using $f_{p,r,s}$, and $SW_{X_{f_{p,r,s}}}(\fs)$ denotes the sum of the Seiberg--Witten invariants of $X_{f_{p,r,s}}$ with respect to all lifts $\tilde \fs\in\Spinc(X_{f_{p,r,s}})$ of $\fs$.}

\bigskip
We end this introduction with the viewpoint that this paper is an analog of \cite{Gerig:taming,Gerig:Gromov} by replacing a symplectic form (and its Lagrangian tori) with a near-symplectic form (and its degeneracy circles), and replacing the 3-manifold $S^1\times S^1\times S^1$ with the 3-manifold $S^1\times S^2$. The near-symplectic Gromov invariants of \cite{Gerig:taming} live in $ECH_*(S^1\times S^2)$ and recover the Seiberg--Witten invariants, while the Lagrangian torus invariants of this paper live in $ECH_*(S^1\times S^1\times S^1)$ and recover triples of Seiberg--Witten invariants.

\begin{convention}
All ordinary (co)homology groups implicitly use $\ZZ$ coefficients.
\end{convention}

%%%%%%%%%%%%%%%%%%%%%%%%%%%%%%%%%%%%%%%%%%%%%%%%%%%%%%%%%%
%%%%%%%%%%%%%%%%%%%%%%%%%%%%%%%%%%%%%%%%%%%%%%%%%%%%%%%%%%
\subsection*{Acknowledgements}
The author thanks Peter Kronheimer for help with aspects of monopole Floer theory, Michael Hutchings for help with aspects of ECH, and Denis Auroux for discussions on pseudoholomorphic disks with Lagrangian boundary conditions. This material is based upon work supported by the National Science Foundation under Award \#1803136.

%%%%%%%%%%%%%%%%%%%%%%%%%%%%%%%%%%%%%%%%%%%%%%%%%%%%%%%%%%
%%%%%%%%%%%%%%%%%%%%%%%%%%%%%%%%%%%%%%%%%%%%%%%%%%%%%%%%%%
\section{Review of pseudoholomorphic curve theory}
\label{Review of pseudoholomorphic curve theory}

We briefly introduce most of the terminology and notations that appear in this paper. More details are found in \cite{Hutchings:lectures}.

%%%%%%%%%%%%%%%%%%%%%%%%%%%%%%%%%%%%%%%%%%%%%%%%%%%%%%%%%%
\subsection{Orbits}
\label{Orbits}

Let $(Y,\lambda)$ be a closed contact 3-manifold, oriented by $\lambda\wedge d\lambda>0$, and let $\xi=\ker\lambda$ be its contact structure. With respect to the \textit{Reeb vector field} $R$ determined by $d\lambda(R,\cdot)=0$ and $\lambda(R)=1$, a Reeb orbit is a map $\gamma:\RR/T\ZZ\to Y$ for some $T>0$ with $\gamma'(t)=R(\gamma(t))$, modulo reparametrization. A given Reeb orbit is \textit{nondegenerate} if the linearization of the Reeb flow around it does not have 1 as an eigenvalue, in which case the eigenvalues are either on the unit circle (such $\gamma$ are \textit{elliptic}) or on the real axis (such $\gamma$ are \textit{hyperbolic}). Assume from now on that $\lambda$ is \textit{nondegenerate}, i.e. all Reeb orbits are nondegenerate, which is a generic property. 

An \textit{orbit set} is a finite set of pairs $\Theta=\lbrace(\Theta_i,m_i)\rbrace$ where the $\Theta_i$ are distinct embedded Reeb orbits and the $m_i$ are positive integers (which may be empty). An orbit set is $\textit{admissible}$ if $m_i=1$ whenever $\Theta_i$ is hyperbolic. Its homology class is defined by
$$[\Theta]:=\sum_im_i[\Theta_i]\in H_1(Y)$$
For a given $\Gamma\in H_1(Y)$, the \textit{ECH chain complex} $ECC_*(Y,\lambda,J,\Gamma)$ is freely generated over $\ZZ/2$ by admissible orbit sets representing $\Gamma$. The differential $\partial_\text{ECH}$ will be defined momentarily.

%%%%%%%%%%%%%%%%%%%%%%%%%%%%%%%%%%%%%%%%%%%%%%%%%%%%%%%%%%
\subsection{Curves}
\label{Curves}

Given two contact manifolds $(Y_\pm,\lambda_\pm)$, possibly disconnected or empty, a \textit{strong symplectic cobordism} from $(Y_+,\lambda_+)$ to $(Y_-,\lambda_-)$ is a compact symplectic manifold $(X,\omega)$ with oriented boundary
$$\partial X=Y_+\sqcup -Y_-$$
such that $\omega|_{Y_\pm}=d\lambda_\pm$.
We can always find neighborhoods $N_\pm$ of $Y_\pm$ in $X$ diffeomorphic to $(-\e,0]\times Y_+$ and $[0,\e)\times Y_-$, such that $\omega|_{N_\pm}=d(e^{\pm s}\lambda_\pm)$ where $s$ denotes the coordinate on $(-\e,0]$. We then glue symplectization ends to $X$ to obtain the \textit{completion}
$$\overline X:=\big((-\infty,0]\times Y_-\big)\cup_{Y_-}X\cup_{Y_+}\big([0,\infty)\times Y_+\big)$$
of $X$, a noncompact symplectic 4-manifold whose symplectic form is also denoted by $\omega$. We will also use the notation $\overline X$ to denote the symplectization $\RR\times Y$ of $(Y,\lambda)$, with $\omega=d(e^s\lambda)$.

An almost complex structure $J$ on a symplectization $(\RR\times Y,d(e^s\lambda))$ is \textit{symplectization-admissible} if it is $\RR$-invariant; $J(\partial_s)=R$; and $J(\xi)\subseteq\xi$ such that $d\lambda(v,Jv)\ge 0$ for $v\in\xi$. An almost complex structure $J$ on the completion $\overline X$ is \textit{cobordism-admissible} if it is $\omega$-compatible on $X$ and agrees with symplectization-admissible almost complex structures on the ends $[0,\infty)\times Y_+$ and $(-\infty,0]\times Y_-$.

\bigskip
Given a cobordism-admissible $J$ on $\overline X$ and orbit sets $\Theta^+=\lbrace(\Theta^+_i,m^+_i)\rbrace$ in $Y_+$ and $\Theta^-=\lbrace(\Theta^-_j,m^-_j)\rbrace$ in $Y_-$, a \textit{$J$-holomorphic curve $\cC$ in $\overline X$ from $\Theta^+$ to $\Theta^-$} is defined as follows. It is a $J$-holomorphic map $\cC\to \overline X$ whose domain is a possibly disconnected punctured compact Riemann surface, defined up to composition with biholomorphisms of the domain, with positive ends of $\cC$ asymptotic to covers of $\Theta^+_i$ with total multiplicity $m^+_i$, and with negative ends of $\cC$ asymptotic to covers of $\Theta^-_j$ with total multiplicity $m^-_j$ (see \cite{Hutchings:lectures}*{\S3.1}). The moduli space of such curves is denoted by $\mM(\Theta^+,\Theta^-)$, but where two such curves are considered equivalent if they represent the same current in $\overline X$, and in the case of a symplectization $\overline X=\RR\times Y$ the equivalence includes translation of the $\RR$-coordinate. An element $\cC\in\mM(\Theta^+,\Theta^-)$ can thus be viewed as a finite set of pairs $\lbrace(C_k,d_k)\rbrace$ or formal sum $\sum d_kC_k$, where the $C_k$ are distinct irreducible somewhere-injective $J$-holomorphic curves and the $d_k$ are positive integers. 

Let $H_2(X,\Theta^+,\Theta^-)$ be the set of relative 2-chains $\Sigma$ in $X$ such that
$$\partial \Sigma=\sum_im^+_i\Theta^+_i-\sum_jm^-_j\Theta^-_j$$
modulo boundaries of 3-chains. It is an affine space over $H_2(X)$, and every curve $\cC$ defines a relative class $[\cC]\in H_2(X,\Theta^+,\Theta^-)$.

%%%%%%%%%%%%%%%%%%%%%%%%%%%%%%%%%%%%%%%%%%%%%%%%%%%%%%%%%%
\subsection{Homology}
\label{Homology}

The \textit{ECH index} $I(\cC)$ of a current $\cC\in\mM(\Theta^+,\Theta^-)$ is an integer depending only on its relative class in $H_2(X,\Theta^+,\Theta^-)$, and is the local expected dimension of this moduli space of $J$-holomorphic currents (see \cite{Hutchings:lectures}*{\S3}). Denote by $\mM_I(\Theta^+,\Theta^-)$ the subset of elements in $\mM(\Theta^+,\Theta^-)$ that have ECH index $I$. 

Given admissible orbit sets $\Theta^\pm$ of $(Y,\lambda)$, the coefficient $\langle\partial_\text{ECH}\Theta^+,\Theta^-\rangle\in\ZZ/2$ is the count (modulo 2) of elements in $\mM_1(\Theta^+,\Theta^-)$ on the symplectization $\overline X=\RR\times Y$. If $J$ is generic then $\partial_\text{ECH}$ is well-defined and $\partial^2_\text{ECH}=0$. The resulting homology is independent of the choice of $J$, depends only on $\xi$ and $\Gamma$, and is denoted by $ECH_*(Y,\xi,\Gamma)$.

The total sum
$$ECH_*(Y,\xi):=\bigoplus_{\Gamma\in H_1(Y)}ECH_*(Y,\xi,\Gamma)$$
has an absolute grading by homotopy classes of oriented 2-plane fields on $Y$ (see \cite{Hutchings:revisited}*{\S 3}), the set of which is denoted by $J(Y)$, and there is a transitive $\ZZ$-action on $J(Y)$. With that said, for the cases relevant to this paper, $ECH_*(Y,\xi,\Gamma)$ has a relative $\ZZ$ grading which is refined by the absolute grading and satisfies
$$|\Theta^+|-|\Theta^-|= I(\cC)$$
for any $\cC\in \mM(\Theta^+,\Theta^-)$.

\bigskip
Given a closed minimal symplectic 4-manifold $(X,\omega)$ so that $Y_\pm=\varnothing$, a class $A\in H_2(X)$, a collection $\bar z\subset X$ of $\frac12I(A)$ distinct points, and a generic $\omega$-compatible $J$ on $X$, let $\mM_{I(A)}(\varnothing,\varnothing;A,\bar z)$ denote the subset of elements in $\mM_{I(A)}(\varnothing,\varnothing)$ on $X$ which represent $A$ and intersect all points $\bar z$. Then the \textit{Gromov invariant} $Gr_{X,\omega}(A)$ is a ``suitably'' weighted count (modulo 2) of elements in $\mM_{I(A)}(\varnothing,\varnothing;A,\bar z)$, where the quantifier ``suitably'' will not be clarified here but is found in \cite{Taubes:counting}*{\S2}. The resulting integer is independent of the choice of $(J,\bar z)$. For example, $Gr_{X,\omega}(0)=1$.

%%%%%%%%%%%%%%%%%%%%%%%%%%%%%%%%%%%%%%%%%%%%%%%%%%%%%%%%%%
\subsection{L-flat approximations}
\label{L-flat approximations}

The symplectic action of an orbit set $\Theta=\lbrace(\Theta_i,m_i)\rbrace$ is defined by
$$\aA(\Theta):=\sum_im_i\int_{\Theta_i}\lambda$$
The symplectic action induces a filtration on the ECH chain complex. For a positive real number $L$, the $L$-\textit{filtered ECH} is the homology of the subcomplex $ECC_*^L(Y,\lambda,J,\Gamma)$ spanned by admissible orbit sets of action less than $L$. The ordinary ECH is recovered by taking the direct limit over $L$, via maps induced by inclusions of the filtered chain complexes.

For a fixed $L>0$ it is convenient (and possible) to modify $\lambda$ and $J$ on small tubular neighborhoods of all Reeb orbits of action less than $L$, in order to relate $J$-holomorphic curves to Seiberg--Witten theory most easily. The desired modifications of $(\lambda,J)$ are called \textit{L-flat approximations}, and were introduced by Taubes in \cite{Taubes:ECH=SWF1}*{Appendix}. They induce isomorphisms on the $L$-filtered ECH chain complex, and the key fact here is that $L$-flat orbit sets are in bijection with Seiberg--Witten solutions of ``energy'' less than $2\pi L$ (see Section~\ref{Taubes' isomorphisms}).

%%%%%%%%%%%%%%%%%%%%%%%%%%%%%%%%%%%%%%%%%%%%%%%%%%%%%%%%%%
%%%%%%%%%%%%%%%%%%%%%%%%%%%%%%%%%%%%%%%%%%%%%%%%%%%%%%%%%%
\section{Review of gauge theory}
\label{Review of gauge theory}

We briefly introduce most of the terminology and notations that appear in this paper. More details are found in \cite{KM:book, HT:Arnold2}.

%%%%%%%%%%%%%%%%%%%%%%%%%%%%%%%%%%%%%%%%%%%%%%%%%%%%%%%%%%
\subsection{Contact 3-manifolds}
\label{Contact 3-manifolds}

Let $(Y,\lambda)$ be a closed oriented connected contact 3-manifold, and choose an almost complex structure $J$ on $\xi$ that induces a symplectization-admissible almost complex structure on $\RR\times Y$. There is a compatible metric $g$ on $Y$ such that $|\lambda|=1$ and $*\lambda=\frac12d\lambda$, with $g(v,w)=\frac12d\lambda(v,Jw)$ for $v,w\in\xi$.

View a spin-c structure $\fs\in\Spinc(Y)$ on $Y$ as an isomorphism class of a pair $(\SS,\cl)$ consisting of a rank 2 Hermitian vector bundle $\SS\to Y$ (the \textit{spinor bundle}) and Clifford multiplication $\cl:TY\to\End(\SS)$. The contact structure $\xi$ (and more generally, any oriented 2-plane field on $Y$) picks out a canonical spin-c structure $\fs_\xi=(\SS_\xi,\cl)$ with $\SS_\xi=\underline{\CC}\oplus\xi$, where $\underline{\CC}\to Y$ denotes the trivial line bundle, and $\cl$ is defined as follows. Given an oriented orthonormal frame $\lbrace e_1,e_2,e_3\rbrace$ for $T_yY$ such that $\lbrace e_2,e_3\rbrace$ is an oriented orthonormal frame for $\xi_y$, then in terms of the basis $(1,e_2)$ for $\SS_\xi$,
$$\cl(e_1)=\bigl( \begin{smallmatrix}
i&0\\ 0&-i
\end{smallmatrix} \bigr),\indent \cl(e_2)=\bigl( \begin{smallmatrix}
0&-1\\ 1&0
\end{smallmatrix} \bigr),\indent \cl(e_3)=\bigl( \begin{smallmatrix}
0&i\\ i&0
\end{smallmatrix} \bigr)$$
There is then a canonical isomorphism
$$H^2(Y)\to\Spinc(Y)$$
where the 0 class corresponds to $\fs_\xi$. Specifically, there is a canonical decomposition $\SS=E\oplus \xi E$ into $\pm i$ eigenbundles of $\cl(\lambda)$, where $E\to Y$ is the complex line bundle corresponding to a given class in $H^2(Y)$.

A \textit{spin-c connection} is a connection $\bA$ on $\SS$ which is compatible with Clifford multiplication in the sense that
$$\nabla_\bA(\cl(v)\psi)=\cl(\nabla v)\psi+\cl(v)\nabla_\bA\psi$$
where $\nabla v$ denotes the covariant derivative of $v\in TY$ with respect to the Levi-Civita connection. Such a connection is equivalent to a Hermitian connection (also denoted by $\bA$) on $\det\SS$, and determines a \textit{Dirac operator}
$$D_\bA:\Gamma(\SS)\stackrel{\nabla_\bA}{\longrightarrow}\Gamma(T^* Y\otimes\SS)\stackrel{\cl}{\longrightarrow}\Gamma(\SS)$$
With respect to the decomposition $\SS=E\oplus\xi E$, there is a unique connection $A_\xi$ on $\xi$ such that its Dirac operator kills the section $(1,0)\in\Gamma(\SS_\xi)$, and there is a canonical decomposition
$$\bA=A_\xi+2A$$
on $\det\SS=\xi E^2$ with Hermitian connection $A$ on $E$. The gauge group $C^\infty(Y,S^1)$ acts on a given pair $(A,\psi)$ by
$$u\cdot(A,\psi)=(A-u^{-1}du,u\psi)$$
In this paper, a \textit{configuration} $\fc$ refers to a gauge-equivalence class of such a pair.

Fix a suitably generic exact 2-form $\mu\in\Omega^2(Y)$ as described in \cite{HT:Arnold2}*{\S 2.2}, and a positive real number $r\in\RR$. A configuration $\fc$ solves \textit{Taubes' perturbed Seiberg--Witten equations} when
\begin{equation}
\label{SW3}
D_\bA\psi=0,\indent\indent* F_A=r(\tau(\psi)-i\lambda)-\frac12* F_{A_\xi}+i*\mu
\end{equation}
where $F_{A_\xi}$ is the curvature of $A_\xi$ and $\tau:\SS\to iT^* Y$ is the quadratic bundle map
$$\tau(\psi)(\cdot)=\langle\cl(\cdot)\psi,\psi\rangle$$
An appropriate change of variables recovers the usual Seiberg--Witten equations (with perturbations) that appear in \cite{KM:book}. 

\begin{remark}
We have suppressed additional ``abstract tame perturbations'' to these equations required to obtain transversality of the moduli spaces of its solutions (see \cite{KM:book}*{\S 10}), because they do not interfere with the analysis presented in this paper. This is further clarified in \cite{HT:Arnold2}*{\S 2.1} and \cite{Taubes:ECH=SWF1}*{\S3.h Part 5}, where the same suppression occurs.
\end{remark}

Denote by $\fM(Y,\fs)$ the set of solutions to~\eqref{SW3}, called \textit{(SW) monopoles}. A solution is \textit{reducible} if its $\Gamma(\SS)$-component vanishes, and is otherwise \textit{irreducible}. The monopoles freely generate the monopole Floer chain complex $\Cfrom^*(Y,\lambda,\fs,J,r)$. The chain complex differential will not be reviewed here. Denote by $\Cfrom_L^*(Y,\lambda,\fs,J,r)$ the submodule generated by irreducible monopoles $\fc$ with energy
$$E(\fc):=i\int_Y\lambda\wedge F_A<2\pi L$$
When $r$ is sufficiently large, $\Cfrom_L^*(Y,\lambda,\fs,J,r)$ is a subcomplex of $\Cfrom^*(Y,\lambda,\fs,J,r)$ and its homology $\Hfrom_L^*(Y,\lambda,\fs,J,r)$ is well-defined and independent of $r$ and $\mu$. Taking the direct limit over $L>0$, we recover the ordinary $\Hfrom^*(Y,\fs)$ in \cite{KM:book} which is independent of $\lambda$ and $J$. It is sometimes convenient to consider the group
$$\Hfrom^*(Y):=\bigoplus_{\fs\in\Spinc(Y)}\Hfrom^*(Y,\fs)$$
over all spin-c structures at once.

%%%%%%%%%%%%%%%%%%%%%%%%%%%%%%%%%%%%%%%%%%%%%%%%%%%%%%%%%%
\subsection{Symplectic cobordisms}
\label{Symplectic cobordisms}

Let $(X,\omega)$ be a strong symplectic cobordism between (possibly disconnected or empty) closed oriented contact 3-manifolds $(Y_\pm,\lambda_\pm)$. Let $\widehat\omega$ denote the particular symplectic form, as specified in \cite{HT:Arnold2}*{\S4.2}, that extends $\omega$ over the cylindrical completion $\overline X$. Choose a cobordism-admissible almost complex structure $J$ on $(\overline X,\widehat\omega)$.

The 4-dimensional gauge-theoretic scenario is analogous to the 3-dimensional scenario. View a spin-c structure $\fs$ on $X$ as an isomorphism class of a pair $(\SS,\cl)$ consisting of a Hermitian vector bundle $\SS=\SS_+\oplus\SS_-$, where the spinor bundles $\SS_\pm$ have rank 2, and Clifford multiplication $\cl:TX\to\End(\SS)$ such that $\cl(v)$ exchanges $\SS_+$ and $\SS_-$ for each $v\in TX$. The set $\Spinc(X)$ of spin-c structures is an affine space over $H^2(X)$, and we denote by $c_1(\fs)$ the first Chern class of $\det\SS_+=\det\SS_-$. A spin-c connection on $\SS$ is equivalent to a Hermitian connection $\bA$ on $\det\SS_+$ and defines a Dirac operator $D_\bA:\Gamma(\SS_\pm)\to\Gamma(\SS_\mp)$.

A spin-c structure $\fs$ on $X$ restricts to a spin-c structure $\fs|_{Y_\pm}$ on $Y_\pm$ with spinor bundle $\SS_{Y_\pm}:=\SS_+|_{Y_\pm}$ and Clifford multiplication $\cl_{Y_\pm}(\cdot):=\cl(v)^{-1}\cl(\cdot)$, where $v$ denotes the outward-pointing unit normal vector to $Y_+$ and the inward-pointing unit normal vector to $Y_-$. There is a canonical way to extend $\fs$ over $\overline X$, and the resulting spin-c structure is also denoted by $\fs$. There is a canonical decomposition $\SS_+=E\oplus K^{-1}E$ into $\mp2i$ eigenbundles of $\cl_+(\widehat\omega)$, where $K$ is the canonical bundle of $(\overline X,J)$ and $\cl_+:\bigwedge^2_+T^*\overline X\to\End(\SS_+)$ is the projection of Clifford multiplication onto $\End(\SS_+)$. This agrees with the decomposition of $\SS_{Y_\pm}$ on the ends of $\overline X$.

The symplectic form $\omega$ picks out the canonical spin-c structure $\fs_\omega=(\SS_\omega,\cl)$, namely that for which $E$ is trivial, and the $H^2(X)$-action on $\Spinc(X)$ becomes a canonical isomorphism. There is a unique connection $A_{K^{-1}}$ on $K^{-1}$ such that its Dirac operator annihilates the section $(1,0)\in\Gamma((\SS_\omega)_+)$, and we henceforth identify a spin-c connection with a Hermitian connection $A$ on $E$.

In this paper, a \textit{configuration} $\fd$ refers to a gauge-equivalence class of a pair $(\bA,\Psi)$ under the gauge group $C^\infty(X,S^1)$-action. A connection $\bA$ on $\det\SS_+$ is in \textit{temporal gauge} on the ends of $\overline X$ if
$$\nabla_\bA=\frac{\partial}{\partial s}+\nabla_{\bA(s)}$$
on $(-\infty,0]\times Y_-$ and $Y_+\times[0,\infty)$, where $\bA(s)$ is a connection on $\det\SS_{Y_\pm}$ depending on $s$. Connections are placed into temporal gauge by an appropriate gauge transformation.

Fix suitably generic exact 2-forms $\mu_\pm\in\Omega^2(Y_\pm)$ as in Section~\ref{Contact 3-manifolds}, a suitably generic exact 2-form $\mu\in\Omega^2(\overline X)$ that agrees with $\mu_\pm$ on the ends of $\overline X$ (with $\mu_*$ denoting its self-dual part), and a positive real number $r\in\RR$. \textit{Taubes' perturbed Seiberg--Witten equations} for a configuration $\fd$ are
\begin{equation}
\label{SW4}
D_\bA\Psi=0,\;\;F^+_A=\frac{r}{2}(\rho(\Psi)-i\widehat\omega)-\frac12F^+_{A_{K^{-1}}}+i\mu_*
\end{equation}
where $F_A^+$ is the self-dual part of the curvature of $A$ and $\rho:\SS_+\to\bigwedge^2_+T^* X$ is the quadratic bundle map
$$\rho(\Psi)(\cdot,\cdot)=-\frac12\big\langle[\cl(\cdot),\cl(\cdot)]\Psi,\Psi\big\rangle$$
Similarly to the 3-dimensional equations, there are additional ``abstract tame perturbations'' which have been suppressed (see \cite{KM:book}*{\S24.1}). Given monopoles $\fc_\pm$ on $Y_\pm$, denote by $\fM(\fc_-,X,\fc_+;\fs)$ the set of solutions to~\eqref{SW4} which are asymptotic to $\fc_\pm$ (in temporal gauge on the ends of $\overline X$), called \textit{SW solutions}.

Similarly to ECH, an ``index'' is associated with each SW solution, namely the local expected dimension of the moduli space of SW solutions. Denote by $\fM_k(\fc_-,X,\fc_+;\fs)$ the subset of elements in $\fM(\fc_-,X,\fc_+;\fs)$ that have index $k$.

%%%%%%%%%%%%%%%%%%%%%%%%%%%%%%%%%%%%%%%%%%%%%%%%%%%%%%%%%%
\subsection{Kronheimer--Mrowka's formalism}
\label{Kronheimer--Mrowka's formalism}

The previous sections concerned the setup of Seiberg--Witten theory from the point of view of symplectic geometry, using Taubes' large perturbations. We now briefly review some relevant aspects of Seiberg--Witten theory from the point of view of Kronheimer--Mrowka's monopole Floer homology.

Let $\bB(Y,\fs)$ denote the space of configurations $[\bA,\psi]$. Since we are not taking large perturbations to the Seiberg--Witten equations, we have to deal with the reducible locus $\bB^\text{red}(Y,\fs)$ which prevents $\bB(Y,\fs)$ from being a Banach manifold. This is done by forming the \textit{blow-up} $\bB^\sigma(Y,\fs)$, the space of configurations $[\bA,s,\psi]$ such that $s\in\RR^{\ge0}$ and $\|\psi\|_2=1$, equipped with the blow-down map
$$\bB^\sigma(Y,\fs)\to\bB(Y,\fs),\indent[\bA,s,\psi]\mapsto[\bA,s\psi]$$
This is a Banach manifold whose boundary $\partial\bB^\sigma(Y,\fs)$ consists of reducible configurations (where $s=0$). The same setup applies to the case that $X$ is a closed 4-manifold. The integral cohomology ring $H^*(\bB^\sigma(M,\fs))$, for $M$ either $Y$ or $X$, is isomorphic to the graded algebra $\big(\Lambda^* H_1(M)/\text{Torsion}\big)\otimes\ZZ[U]$, where $U$ is a 2-dimensional generator (see \cite{KM:book}*{Proposition 9.7.1}).

We can construct a certain vector field $\vV^\sigma$ on $\bB^\sigma(Y,\fs)$ using the pull-back of the gradient of the Chern-Simons-Dirac functional $\lL_\text{CSD}:\bB(Y,\fs)\to\RR$ (see \cite{KM:book}*{\S4.1}). Strictly speaking, the Chern-Simons-Dirac functional is not well-defined on $\bB(Y,\fs)$ unless $c_1(\fs)$ is torsion, but such spin-c structures are the only ones relevant to this paper. Likewise, the perturbed gradient $\grad\lL_\text{CSD}+\fq$ gives rise to a vector field $\vV^\sigma+\fq^\sigma$, where $\fq$ is an ``abstract tame perturbation'' (see \cite{KM:book}*{\S 10}). We always assume that $\fq$ is chosen from a Banach space of tame perturbations so that all stationary points of $\vV^\sigma+\fq^\sigma$ are nondegenerate.

The critical points (i.e. stationary points) of $\vV^\sigma+\fq^\sigma$ are either \textit{irreducibles} of the form $[\bA,s,\psi]$ with $s>0$ and $[\bA,s\psi]\in\crit(\grad\lL_\text{CSD}+\fq)$, or \textit{reducibles} of the form $[\bA,0,\psi]$ with $\psi$ an eigenvector of $D_\bA$. We can package these critical points together\footnote{To define the Floer groups over $\ZZ$ we must also equip the critical points with choices of ``coherent orientations''.} in various ways to form the monople Floer (co)homologies, such as $\Hto_*(Y,\fs)$ and $\Hfrom^*(Y,\fs)$. The differentials will not be reviewed here, but we do assume in this paper that all perturbations $\fq$ are chosen so that the differentials are well-defined.

\begin{remark}
If $\fq$ is one of Taubes' sufficiently large perturbations associated with a contact form (given in Section~\ref{Contact 3-manifolds}), then the image of the set of critical points under the blow-down map is $\fM(Y,\fs)$. In fact, we no longer need to use the blow-up model.
\end{remark}

Let $X$ either be a closed 4-manifold or have boundary $Y$. There is a partially-defined restriction map $r:\bB^\sigma(X,\fs)\dashrightarrow\bB^\sigma(Y,\fs)$ whose domain consists of those configurations $[\bA,s,\Psi]$ satisfying $\Psi_Y:=\Psi|_Y\ne0$, such that
$$r([\bA,s,\Psi])=\Big[\bA|_Y,s\|\Psi_Y\|_2,\Psi_Y/\|\Psi_Y\|_2\Big]$$
If $X=[0,1]\times Y$ then there is a family of restriction maps $r_t:\bB^\sigma(X,\fs)\dashrightarrow\bB^\sigma(Y,\fs)$ for $t\in[0,1]$.

With respect to a cylindrical completion $\overline X$ of $X$, the unperturbed Seiberg--Witten equations on $\bB(\overline X,\fs)$ take the form
\begin{equation}
\label{SWblowup}
D_\bA\Psi=0,\;\;F^+_\bA=s^2\frac{1}{4}\rho(\Psi)
\end{equation} 
on $\bB^\sigma(\overline X,\fs)$. In the cylindrical case $\overline X=\RR\times Y$ with spin-c structure induced from $\fs$ on $Y$ and cylindrical perturbation $\fp$, any solution $\fd$ to the $\fp^\sigma$-perturbed version of~\eqref{SWblowup} on $\RR\times Y$ determines a path
$$\check\fd(t):=r_t(\fd)\in\bB^\sigma(Y,\fs)$$
because there is a unique continuation theorem which ensures that $r_t$ is defined on each slice $\fd|_{\lbrace t\rbrace\times Y}$ (see \cite{KM:book}*{\S10.8}).

In the general case of a cobordism $(X,\fs):(Y_+,\fs_+)\to(Y_-,\fs_-)$, we fix abstract perturbations $\fq_\pm$ on $Y_\pm$ and extend them to a suitable abstract perturbation $\fp$ on $\overline X$. Given critical points $\fc_\pm$ over $(Y_\pm,\fs_\pm)$, we denote by $M(\fc_-,\fc_+;\fs)$ the subset of $\fp^\sigma$-perturbed Seiberg--Witten solutions $\fd\in\bB^\sigma(\overline X,\fs)$ for which $\check\fd$ (on the ends of $\overline X$) is asymptotic to $\fc_\pm$ as $t\to\pm\infty$. Depending on the context, we may alternatively write $M(\fc_-,X,\fc_+;\fs)$ to make the manifold explicit. Note that $M(\fc_-,X,\fc_+;\fs)=\fM(\fc_-,X,\fc_+;\fs)$ when using Taubes' perturbations in Section~\ref{Symplectic cobordisms}.

%%%%%%%%%%%%%%%%%%%%%%%%%%%%%%%%%%%%%%%%%%%%%%%%%%%%%%%%%%
\subsection{Closed 4-manifolds}
\label{Closed 4-manifolds}

The case $Y_\pm=\varnothing$ recovers Seiberg--Witten theory on closed oriented 4-manifolds $(X,g)$. For a generic choice of abstract perturbation $\mu\in\Omega^2_+(X)$ to~\eqref{SWblowup}, denote the space of solutions by $\fM(\fs)$. When $b^2_+(X)>0$, a generic choice of $\mu$ makes $\fM(\fs)$ a finite-dimensional compact orientable smooth manifold, where the orientation is determined by a \textit{homology orientation} of $X$, this being an orientation of $\det H^1(X;\RR)\otimes\det H^2_+(X;\RR)$ (see also Appendix~\ref{appendix}).

The dimension of $\fM(\fs)$ is equal to the integer
$$d(\fs):=\frac{1}{4}\left(c_1(\fs)^2-2\chi(X)-3\sigma(X)\right)$$
where $c_1(\fs)$ denotes the first Chern class of the spin-c structure's positive spinor bundle, $\chi(X)$ denotes the Euler characteristic of $X$, and $\sigma(X)$ denotes the signature of $X$. If $\dim\fM(\fs)<0$, then $\fM(\fs)$ is empty and the \textit{Seiberg--Witten invariant} is defined to be zero. In the remaining cases, the moduli space gives a well-defined element $[\fM(\fs)]\in H_*(\bB^\sigma(X,\fs))$. 

\begin{definition}
\label{defn:SW}
For a given choice of homology orientation of $X$, the \textit{Seiberg--Witten invariant} $SW_X(\fs)\in\ZZ$ is defined as follows. When $d(\fs)\ge0$ is even its value is
$$SW_X(\fs):=\left\langle U^{d(\fs)/2},[\fM(\fs)]\right\rangle\in\ZZ$$
and it is defined to be zero if $d(\fs)$ is odd.
\end{definition}

%%%%%%%%%%%%%%%%%%%%
\subsubsection{Choice of ``chamber''}

When $b^2_+(X)>1$, the value of the Seiberg--Witten invariant is a diffeomorphism invariant of $X$ independent of the choice of generic pairs $(g,\mu)\in\operatorname{Met}(X)\times\Omega^2_+(X)$, where $\operatorname{Met}(X)$ denotes the Fr\'echet space of smooth Riemannian metrics on $X$. When $b^2_+(X)=1$, there is a ``wall-crossing phenomenon'' as follows. Denote by $\omega_g$ the unique (up to scalar multiplication) nontrivial self-dual harmonic 2-form with respect to $g$. The set of pairs $(g,\mu)$ satisfying the constraint
\begin{equation}
\label{chamber}
2\pi[\omega_g]\cdot c_1(\fs)+\int_X\omega_g\wedge\mu=0
\end{equation}
defines a ``wall'' which separates $\operatorname{Met}(X)\times\Omega^2_+(X)$ into two open sets, called \textit{$c_1(\fs)$-chambers}. The Seiberg--Witten invariant is constant on any $c_1(\fs)$-chamber, and the difference between chambers is computable.

A symplectic form $\omega$ on $X$ picks out a canonical $c_1(\fs)$-chamber, namely those pairs $(g,\mu)$ for which the left hand side of~\eqref{chamber} is negative. This is the chamber that pertains to the large $r$ version of Taubes' perturbed Seiberg--Witten equations~\eqref{SW4}.

%%%%%%%%%%%%%%%%%%%%%%%%%%%%%%%%%%%%%%%%%%%%%%%%%%%%%%%%%%
\subsection{Gradings}
\label{Gradings}

The groups $\Hfrom^*(Y)$ and $\Hto_*(Y)$ have an absolute grading by homotopy classes of oriented 2-plane fields on $Y$ (see \cite{KM:book}*{\S 28} and \cite{Hutchings:revisited}*{\S 3}), the set of which is denoted by $J(Y)$, and there is a transitive $\ZZ$-action on $J(Y)$. This grading of a critical point $\fc$ is denoted by $|\fc|\in J(Y)$. It is useful to write out the induced relative $\ZZ$ grading when $\fs$ is torsion (the cases relevant to this paper), as follows. Given critical points $\fc_\pm$ over $(Y,\fs)$, each trajectory $\fd\in M(\fc_-,\fc_+;\fs)$ over $\RR\times Y$ has a Fredholm operator $Q_\fd$ which, roughly speaking, is the linearization of the perturbed version of~\eqref{SWblowup} and the gauge group action (see \cite{KM:book}*{\S14.4}). The \textit{relative grading} $\gr(\fc_-,\fc_+)$ between $\fc_-$ and $\fc_+$ is defined to be the Fredholm index of $Q_\fd$ for any $\fd\in M(\fc_-,\fc_+;\fs)$, and
$$|\fc_+|-|\fc_-|=\gr(\fc_-,\fc_+)$$
as expected. The fact that this index does not depend on the choice of $\fd$ (for $\fs$ torsion) follows immediately from \cite{KM:book}*{Proposition 14.4.5, Lemma 14.4.6}.

As explained in \cite{KM:book}*{\S22.5}, after choosing an orientation of the vector space $H^1(Y;\RR)$ there is a canonical isomorphism
$$\Hfrom^j(Y)\cong\Hto_{-j}(-Y)$$
for $j\in J(Y)$, noting that an oriented 2-plane field on $Y$ is also an oriented 2-plane field on $-Y$.

%%%%%%%%%%%%%%%%%%%%%%%%%%%%%%%%%%%%%%%%%%%%%%%%%%%%%%%%%%
%%%%%%%%%%%%%%%%%%%%%%%%%%%%%%%%%%%%%%%%%%%%%%%%%%%%%%%%%%
\section{Taubes' isomorphisms}
\label{Taubes' isomorphisms}

With $\ZZ/2$ coefficients, there is a canonical isomorphism of relatively graded modules
\begin{equation}
\label{eqn:Taubes}
ECH_*(Y,\xi,\Gamma)\cong\Hfrom^{-*}(Y,\fs_\xi+\PD(\Gamma))
\end{equation}
which also preserves the absolute gradings by homotopy classes of oriented 2-plane fields. This isomorphism is constructed on the $L$-filtered chain level. 

\begin{theorem}[\cite{Taubes:ECH=SWF1}*{Theorem 4.2}]
\label{thm:generators}
Fix $L>0$ and a generic $L$-flat pair $(\lambda,J)$ on $(Y,\xi)$. Then for $r$ sufficiently large and $\Gamma\in H_1(Y)$, there is a canonical bijection from the set of generators of $ECC^L_*(Y,\lambda,\Gamma,J)$ to the set of generators of $\Cfrom_L^*(Y,\lambda,\fs_\xi+\PD(\Gamma),J,r)$.

The image of an admissible orbit set $\Theta$ under this bijection will be denoted by $\fc_\Theta$, and is an irreducible SW monopole that solves Taubes' perturbed Seiberg--Witten equations~\eqref{SW3}.
\end{theorem}

There is a distinguished element in both $ECH_0(Y,\xi_0,0)$ and $\Hto_0(-Y,\fs_\xi)$, the \textit{contact invariant} $[\varnothing]$ induced by $\xi_0$, and Taubes' isomorphism preserves it \cite{Taubes:ECH=SWF5}*{Theorem 1.1}. On the chain level the correspondence is $\varnothing\mapsto\fc_\varnothing$.

Likewise, for closed 4-manifolds Taubes constructed the following well-known equivalence of integers, which will be assumed throughout the paper.

\begin{theorem}[\cite{Taubes:Gr=SW}]
For a closed minimal symplectic manifold $(X,\omega)$ and $A\in H_2(X)$,
$$Gr_{X,\omega}(A)=SW_X(\fs_\omega+A)\in\ZZ/2$$
where $\omega$ determines the chamber for defining the Seiberg--Witten invariants when $b^2_+(X)=1$.
\end{theorem}

%%%%%%%%%%%%%%%%%%%%%%%%%%%%%%%%%%%%%%%%%%%%%%%%%%%%
%%%%%%%%%%%%%%%%%%%%%%%%%%%%%%%%%%%%%%%%%%%%%%%%%%%%
\section{Floer homologies of 3-tori}
\label{Floer homologies of 3-tori}

Recall that $T^3=(S^1\times S^1)\times S^1$ is parametrized by $0\le x\le2\pi$ and $0\le y\le2\pi$ and $0\le\theta\le2\pi$. We use the unorthodox orientation of $T^3$ by the 3-form $-dx\,dy\,d\theta$ so that $\lambda_0\wedge d\lambda_0$ is positive (for the contact form $\lambda_0$ specified in Section~\ref{Introduction and main result}), and we orient the three $S^1$-factors by the 1-forms $dx,\,dy,\,-d\theta$. Then using Hom-duality and Poincar\'e-duality, the positive basis of $H_2(T^3)\cong H_1(T^3)\cong\ZZ^3$ is denoted
\begin{align*}
x &:=+[S^1\times\lbrace*\rbrace\times\lbrace*\rbrace]\\
y &:=+[\lbrace*\rbrace\times S^1\times\lbrace*\rbrace]\\
\bar\theta &:=-[\lbrace*\rbrace\times\lbrace*\rbrace\times S^1]
\end{align*}

\begin{remark}
For fixed $(r,s)\in\ZZ^2$ consider the map on $T^3$ given by $f(x,y,\theta)=(x+r\theta,y+s\theta,\theta)$. The induced map on $H_1(T^3)$ is given by $f_*x=x$, $f_*y=y$, $f_*\bar\theta=\bar\theta-rx-sy$. The induced map on $H_2(T^3)$ is given by $f_*x=x+r\bar\theta$, $f_*y=y+s\bar\theta$, $f_*\bar\theta=\bar\theta$.
\end{remark}

The spin-c structure $\fs_0:=\fs_{\xi_0}$ determined by $\xi_0$ is torsion, i.e. $c_1(\fs_0)=c_1(\xi_0)=0$, and so Taubes' isomorphism~\eqref{eqn:Taubes} reads
$$ECH_j(T^3,\xi_0,0)\cong \Hfrom^j(T^3, \fs_0)\cong \Hto_{-j}(T^3,\fs_0)$$
where $j\in J(T^3,\fs_0)\cong\ZZ$ as $\ZZ$-sets. We have made use of the fact that $T^3$ admits an orientation-reversing self-diffeomorphism. There is a unique class $j=[\xi_*]$ represented by an oriented 2-plane field $\xi_*$ on $T^3$ which is invariant under translations, and the $\ZZ$-grading is made absolute by identifying $[\xi_*]=0\in\ZZ$.

\begin{prop}
\label{prop:ECHSWF}
If $\Gamma\in H_1(T^3)$ is not zero then $\Hto_*(T^3,\fs_0+\Gamma)=0$. In the remaining case $\Gamma=0$, $\Hto_*(T^3,\fs_0)$ is zero in gradings above $0$, and for each $n\le0$
$$\Hto_n(T^3,\fs_0)\cong H_2(T^3)\cong\ZZ^3$$
such that this isomorphism commutes with the orientation-preserving diffeomorphisms of $T^3$.
\end{prop}

\begin{proof}
This is precisely \cite{KM:book}*{Proposition 3.10.1} with the remark that our grading conventions are opposite to those in \cite{KM:book}. Alternatively, these group isomorphisms are also established directly on ECH \cite{Hutchings:T3}*{Theorems 1.2 and 1.3} and so we can apply Taubes' isomorphisms~\eqref{eqn:Taubes}. We briefly clarify the assertion about naturality with respect to diffeomorphisms of $T^3$. The monopole Floer groups are modules over $H^*(\bB^\sigma(T^3,\fs_0))\cong\big(\Lambda^* H_1(T^3)/\text{Torsion}\big)\otimes\ZZ[U]$, the isomorphism being natural with respect to the action of the group $\lbrace f\in\Diff_+(T^3)\;|\;f^*\fs_0=\fs_0\rbrace$ \cite{KM:book}*{Proposition 9.7.1}. But $f^*\fs_0=\fs_0$ for all $f\in\Diff_+(T^3)$, because $\fs_0$ is the unique torsion spin-c structure on $T^3$.
\end{proof}

\begin{remark}
As explained in \cite{Hutchings:T3} with $\ZZ$ coefficients, the degree zero generators of the ECH chain complex consist of the empty set and, for each $v\in\ZZ^2$, a generator $h(v)=\lbrace(h_1,1),(h_2,1)\rbrace$ consisting of two hyperbolic orbits satisfying $[h_1]=v$ and $[h_2]=-v$. These generators satisfy the relation $h(v) + h(v') = h(v+v')$. Thus, $ECH_0(T^3,\xi_0,0)$ over $\ZZ/2$ is identified with a copy of $\ZZ/2$ (generated by the empty set) plus a copy of $(\ZZ/2)^2$ (given by the generators $h(v)$ modulo the relation). This can be identified with $H_2(T^3;\ZZ/2)\cong(\ZZ/2)^2\oplus\ZZ/2$ and is natural with respect to contactomorphisms.
\end{remark}

%%%%%%%%%%%%%%%%%%%%%%%%%%%%%%%%%%%%%%%%%%%%%%%%%%%%
%%%%%%%%%%%%%%%%%%%%%%%%%%%%%%%%%%%%%%%%%%%%%%%%%%%%
\section{Invariants of 2-tori}
\label{Invariants of 2-tori}

In this section we build the tentative Lagrangian torus invariants for $L\subset(X,\omega)$ with $X$ minimal, indexed by relative classes $A\in H_2(X_0,UT^* L)$, which effectively count (certain) $J$-holomorphic curves in a completion of $(X_0,\omega,J)$ representing a given relative class $A$. Each invariant is an element of $ECH_*(T^3,\xi_0,\partial A)\cong\Hto_*(T^3,\fs_0+\partial A)$, so in light of Proposition~\ref{prop:ECHSWF} we must assume that $\partial A=0\in H_1(UT^* L)$.

\bigskip
In order to obtain well-defined counts of $J$-holomorphic curves which represent a given relative class $A$, we will need to ensure a bound on their energy as well as a bound on the symplectic action of their orbit sets. As explained in \cite{Hutchings:fieldtheory}, these bounds are given by the quantity
\begin{equation}
\rho(A):=\int_\Sigma\omega+\int_{\partial\Sigma}\lambda_0
\end{equation}
where $u:\Sigma\to X_0$ is any given smooth map which represents $A$, whose domain $\Sigma$ is a compact oriented smooth surface with boundary satisfying $u(\partial \Sigma)\subset \partial X_0$.

Now, three perturbations will be made to $\lambda_0$. First, we will want all Reeb orbits to be nondegenerate in order to define ECH. Second, we will want all Reeb orbits of action less than $\rho(A)$ to be $\rho(A)$-flat in order to relate the $J$-holomorphic curves to Seiberg--Witten theory. Third, we will want the elliptic orbits of action less than $\rho(A)$ to be ``$\rho(A)$-positive'' in order to guarantee transversality of the relevant moduli spaces of $J$-holomorphic curves (specifically, to rule out negative ECH index curves). As defined in \cite{Hutchings:beyond}, the quantifier ``$\rho(A)$-positive'' means the following:

\begin{definition}
Fix $L>0$. Let $\gamma$ be a nondegenerate embedded elliptic orbit with rotation class $\theta\in\RR/\ZZ$ and symplectic action $\aA(\gamma)<L$. Then $\gamma$ is \textit{L-positive} if $\theta\in(0,\aA(\gamma)/L)\mod 1$. Here we note that the linearization of the Reeb flow around $\gamma$ is conjugate to a rotation by angle $2\pi\theta$ with respect to a trivialization $\tau$ of $\gamma^*\xi_0$, and the equivalence class of this rotation number $\theta$ in $\RR/\ZZ$ does not depend on $\tau$.
\end{definition}

\noindent
Such perturbations give us control over the orbits of low symplectic action, at the expense of producing new orbits of high symplectic action with unknown properties. This is sufficient for the purposes of this paper, because for a given class $A$ only the orbit sets of symplectic action less than $\rho(A)$ are relevant to the tentative Lagrangian torus invariant.

\begin{lemma}
\label{lem:nbhd}
For a given $A\in H_2(X_0,UT^* L)$ there is a choice of neighborhood $\nN$ of $L\subset X$ such that $(X-\nN,\omega)$ is a symplectic manifold with contact-type boundary $(T^3,\lambda_A)$. Here, $\lambda_A$ is a nondegenerate contact form with contact structure $\xi_0=\ker\lambda_0$ but whose orbits of symplectic action less than $\rho(A)$ are all $\rho(A)$-flat and are either positive hyperbolic or $\rho(A)$-positive elliptic.
\end{lemma}

To prove this lemma, we pass from $\lambda_0$ to a slightly more general scenario, analyzing a 1-form
\begin{equation}
\label{eqn:form}
\lambda=a_1(\theta)dx + a_2(\theta)dy
\end{equation}
defined by a smooth pair $a=(a_1,a_2):[0,2\pi]\to\RR^2-\lbrace(0,0)\rbrace$. Let $a\times a':=a_1a_2'-a_2a_1'$, where the tick-mark signifies the derivative with respect to $\theta$. The condition for $\lambda$ to be a positive contact form (with respect to our volume form on $T^3$) is then $a\times a'(\theta)>0$ for all $\theta\in[0,2\pi]$. The Reeb field of $\lambda$ is $\frac{1}{a\times a'(\theta)}\left(a_2'(\theta)\frac{\partial}{\partial x}-a_1'(\theta)\frac{\partial}{\partial y}\right)$, and the condition for which $T(\theta_0)\subset T^3$ is a torus foliated by orbits is given by
\begin{equation}
\label{eqn:angle}
\frac{a_1'(\theta_0)}{a_2'(\theta_0)}\in\QQ\cup\lbrace\pm\infty\rbrace
\end{equation}
Every embedded orbit in $T(\theta_0)$ represents the same class in $H_1(T(\theta_0))$ and they all have the same action $\aA(\theta_0)>0$. The next lemma below shows how to perturb these Morse--Bott orbits, in the sense of \cite{Bourgeois:thesis} and adapted from \cite{Hutchings:beyond}*{Lemma 5.4}.

\begin{lemma}
\label{lem:Bourgeois}
Suppose the positive contact form $\lambda=a_1(\theta)dx + a_2(\theta)dy$ satisfies $a'\times a''(\theta_0)>0$ for all $\theta_0\in[0,2\pi]$ that satisfy~\eqref{eqn:angle}. Then for every $L>0$ and sufficiently small $\delta>0$, there exists a perturbation $e^{f_{\delta,L}}\lambda$ of $\lambda$ satisfying the following properties:

	$\bullet$ $f_{\delta,L}\in C^\infty(T^3)$ satisfies $||f_{\delta,L}||_{C^0}<\delta$,

	$\bullet$ Each family of orbits in the torus $T(\theta_0)$ with $\aA(\theta_0)<L$ is replaced by a positive hyperbolic\\
	\indent~~~orbit and an $L$-positive elliptic orbit, both of action less than $L$ and within $\delta$ of $\aA(\theta_0)$,

	$\bullet$ $e^{f_{\delta,L}}\lambda$ has no other embedded orbits of action less than $L$.
\end{lemma}

\begin{proof}
The function $f_{\delta,L}$ is given by Bourgeois' perturbation \cite{Bourgeois:thesis} of $\lambda$, which breaks up each $T(\theta_0)$ into two embedded nondegenerate orbits of action slightly less than $\aA(\theta_0)$ in addition to orbits of action greater than $L$. Namely, there is a positive hyperbolic orbit and an elliptic orbit $e_{\theta_0}$, both representing the same class in $H_1(T(\theta_0))$. For sufficiently small perturbations there cannot exist other orbits of action less than $L$, otherwise we would find a sequence $\lbrace(\gamma_k,\delta_k)\rbrace_{k\in\NN}$ of such orbits of uniformly bounded action $L$ and perturbations $\delta_k\to0$ for which a subsequence converges to one of the original degenerate orbits (by the Arzel\`a--Ascoli theorem), yielding a contradiction.

It remains to compute the rotation class of the elliptic orbit created from each Morse--Bott family. Let $a^\perp:=a_2(\theta)\partial_x-a_1(\theta)\partial_y$. The basis $\langle\partial_\theta,-a^\perp\rangle$ defines a trivialization $\tau$ of the contact structure $\xi$ since
$$d\lambda(\partial_\theta,-a^\perp)=a\times a'(\theta)>0$$
We then compute the Lie derivatives of the Reeb field $R$,
$$\lL_{\partial_\theta}R=\frac{a'\times a''}{(a\times a')^2}(-a^\perp),\indent \lL_{a^\perp}R=0$$
to see that the linearized Reeb flow along $T(\theta_0)$ induces the linearized return map
$$P_{T(\theta_0)}:=\1+\begin{pmatrix}
0&0\\ r(\theta_0)\aA(\theta_0)&0
\end{pmatrix}$$
on $\xi$ in the chosen basis, where $r:=\frac{a'\times a''}{(a\times a')^2}$. The linearized return map along $e_{\theta_0}$ is a perturbation of $P_{T(\theta_0)}$, so the rotation number of $e_{\theta_0}$ has the same sign as $r(\theta_0)$, i.e. it has the sign of $a'\times a''(\theta_0)$. This rotation number can be made arbitrarily small by choosing $\delta$ sufficiently small (copying the same proof of \cite{Gerig:taming}*{Lemma 3.6}), so it follows from the assumption on $a'\times a''(\theta_0)$ that each $e_{\theta_0}$ is $L$-positive.
\end{proof}

\begin{proof}[Proof of Lemma~\ref{lem:nbhd}]
Since $\lambda_0$ satisfies the hypothesis of Lemma~\ref{lem:Bourgeois} ($a'\times a''=1$), we can perturb this Morse--Bott contact form $\lambda_0$ (\`a la Bourgeois) so that all orbits of action less than $\rho(A)$ are nondegenerate and $\rho(A)$-positive when elliptic. The remainder of the proof follows that of \cite{Gerig:taming}*{Lemma 3.9} verbatim.
\end{proof}

Fix a relative class $A$. Thanks to Lemma~\ref{lem:nbhd}, we choose $\nN$ so that $(X_0,\omega)$ is a strong symplectic cobordism from the empty set $(\varnothing,0)$ to the contact 3-manifold $(T^3,\lambda_A)$. Let $\overline{X_0}$ denote its completion, and fix a cobordism-admissible almost complex structure $J$ on $(\overline{X_0},\omega)$. As shown in \cite{Hutchings:fieldtheory}, there are induced \textit{ECH cobordism maps} of the form
\begin{equation}
\label{eqn:chainmap}
ECH_0(\varnothing,0,0)\to ECH_*(T^3,\xi_0,0)
\end{equation}
defined by suitable counts of Seiberg--Witten solutions on $\overline{X_0}$. Since $ECH_0(\varnothing,0,0)\cong\ZZ/2$ is generated by the empty set of orbits, the map~\eqref{eqn:chainmap} should really be viewed as an element of $ECH_*(T^3,\xi_0,0)$. We now present a definition of this element via counts of $J$-holomorphic curves in $\overline{X_0}$.

Choose a nonnegative even integer $I$ and a set of $\frac{I}{2}$ disjoint points $\bar z:=\lbrace z_1,\ldots,z_{I/2}\rbrace\subset X_0$. Denote by $\mM_I(\varnothing,\Theta;A,\bar z)$ the subset of elements in $\mM_I(\varnothing,\Theta)$ which represent the class $A$ and intersect all points $\bar z$. Define the chain
\begin{equation}
\label{eqn:chain}
\sum_\Theta\sum_{\cC\in\mM_I(\varnothing,\Theta;A,\bar z)}q(\cC)\cdot\Theta\in ECC_*(T^3,\lambda_A,0)
\end{equation}
where $\Theta$ indexes over the admissible orbit sets, and where the weights $q(\cC)\in\ZZ/2$ are equal to 1 unless $\cC$ contains an index zero multiply covered torus, in which case the weight depends on the multiplicity of each such torus and is given by \cite{Taubes:counting}*{Definition 3.2}. The following theorem implies that this chain~\eqref{eqn:chain} is well-defined (in particular, the relevant moduli spaces are 0-dimensional compact manifolds) and that the homology class of~\eqref{eqn:chain} does not depend on the choice of $(J,\bar z)$.

\begin{theorem}
\label{thm:class}
Fix $I\in2\NN\cup\lbrace0\rbrace$ and $A\in H_2(X_0,UT^* L)$ such that $\partial A=0$. For generic $J$, the chain~\eqref{eqn:chain} induces a well-defined element
$$\Phi(A,I)=U^{I/2}\circ\Phi(A,0)\in ECH_{g(A,I)}(T^3,\xi_0,0)$$
in a single absolute grading $g(A,I)\in\ZZ$, where $U$ denotes the (degree $-2$) U-map in ECH.
\end{theorem}

\begin{proof}
The fact that the chain~\eqref{eqn:chain} is a cycle whose homology class $\Phi(A,I)$ does not depend on the choice of $\bar z$ follows \cite{Gerig:taming}*{\S3.4} verbatim (where the crucial Lemma~\ref{lem:nbhd} replaces the corresponding \cite{Gerig:taming}*{Lemma 3.9}); same for its decomposition in terms of the U-map. Note that we have made use of the fact that $X$ is minimal here, as multiply-covered $J$-holomorphic exceptional spheres sometimes cause issues with finiteness of $\mM_I(\varnothing,\Theta;A,\bar z)$. The fact that $\Phi(A,I)$ also does not depend on the choice of $J$ follows \cite{Gerig:Gromov}*{\S4} verbatim to relate the construction of the chain~\eqref{eqn:chain} to suitable counts of Seiberg--Witten solutions on $\overline{X_0}$, so that $\Phi(A,I)$ agrees with the image of the ECH cobordism map~\eqref{eqn:chainmap} defined in \cite{Hutchings:fieldtheory}.
\end{proof}

Thanks to the U-map, the elements in Theorem~\ref{thm:class} for a given relative class $A$ are all determined by one such element. The relevant Lagrangian torus invariant will be defined as the element for a specific integer $I$ depending on $A$, such that $g(A,I)=0$. In order to describe this specific integer, we must analyze the decomposition of $X$ into $X_0$ and $\nN$.

\begin{lemma}
\label{lem:lift}
For any class $A\in H_2(X_0,UT^* L)$ whose boundary is zero in $H_1(UT^* L)$, the ECH index $I(\tilde A)$ of a lift $\tilde A\in H_2(X)$ of $A$ is independent of the choice of such lift.

For any spin-c structure $\fs\in\Spinc(X_0)$ whose restriction to its boundary $UT^* L$ is trivial, the Seiberg--Witten index $d(\tilde\fs)$ of a lift $\tilde\fs\in\Spinc(X)$ of $\fs$ is independent of the choice of such lift. 
\end{lemma}

\begin{proof}
The decomposition $X=X_0\cup_\partial\nN$ induces the cohomological Mayer--Vietoris sequence
$$\cdots\to H^1(UT^* L)\stackrel{\mM\vV}{\to} H^2(X)\stackrel{(i_{X_0}^*,i_{\nN}^*)}{\longrightarrow} H^2(X_0)\oplus H^2(\nN)\stackrel{i_{\partial X_0}^*-i_{\partial\nN}^*}{\longrightarrow} H^2(UT^* L)\to\cdots$$
We are only considering the classes $(A,0)\in H_2(X_0,UT^* L)\oplus H_2(\nN,\partial\nN)$ whose boundary is zero in $H_1(UT^* L)$, i.e. live in the kernel of the induced restriction map $i_{\partial X_0}^*-i_{\partial\nN}^*$, so all such classes $A\in H_2(X_0,UT^* L)$ come from the restriction map $\text{Ker}(i_{\nN}^*)\subset H^2(X)\to H^2(X_0)$. But there might not be a unique lift $\tilde A\in H^2(X)$ of $A$, the ambiguity coming from
$$\text{Ker}\big(i_{X_0}^*:\text{Ker}(i_{\nN}^*)\to H^2(X_0)\big)=\text{Im}(\mM\vV)\subset H^2(X)$$
Although not needed, these Mayer--Vietoris classes are dual to multiples of $[L]$ in $H_2(X)$.

In what follows we suppress Poincar\'e--Lefschetz duality and use $K$ to denote both the canonical bundle over $X$ (determined by $\omega$) and its 1st Chern class. Given a lift $\tilde A\in H_2(X)\cong H^2(X)$ we claim that the ECH index satisfies $I(\tilde A)=I(\tilde A+v)$ for every $v\in\text{Im}(\mM\vV)$, which is equivalent to $v\cdot(v+2\tilde A-K)=0$. To show this we use de Rham cohomology, noting that the cup product operation $H^2(X)\otimes H^2(X)\to H^4(X)\cong\ZZ$ vanishes on torsion elements. The support of $v\in H^2(X;\RR)$ is contained in a small tubular neighborhood of $UT^* L\subset X$, so $v\wedge v=0$ because we can take two different representative 2-forms for $v$, one having support in a collar neighborhood of $UT^* L\subset X_0$ and the other having support in a collar neighborhood of $UT^* L\subset\nN$.\footnote{Here is a way to see that $v\cdot v=0$ without using de Rham cohomology. The map $\mM\vV$ is not a graded ring homomorphism, but it is the composition of the suspension isomorphism $H^1(UT^* L)\to H^2\big(\Sigma(UT^* L)\big)$ with the graded ring homomorphism $\mM\vV^*:H^2\big(\Sigma(UT^* L)\big)\to H^2(X)$ induced by the map $X\hookrightarrow CX_0\cup C\nN\twoheadrightarrow\Sigma(UT^* L)$ \cite{MO:cupMV}. Since $v$ is in the image of $\mM\vV$ it can be written as $\mM\vV^*(v^*)$ for some $v^*\in H^2\big(\Sigma(UT^* L)\big)$, and so $v\smile v=\mM\vV^*(v^*\smile v^*)=\mM\vV^*(0)=0$ because the cup product operation is trivial on suspension spaces.} Similarly, $v\cdot\tilde A=v\cdot K=0$ because $\tilde A$ and $K$ have representatives which are supported in $X_0$ and away from $UT^* L$.

We can translate the previous paragraph in terms of spin-c structures, as follows. The symplectic form induces the canonical $H_2(X)$-equivariant isomorphism $\Spinc(X)\cong H^2(X)$, such that the restriction map $\Spinc(X)\to\Spinc(X_0)$ sends $v\in H^2(X)$ to the spin-c structure on $X_0$ associated with the relative class $\PD i_{X_0}^*(v)\in H_2(X_0,UT^* L)$. Under this translation we see that for each spin-c structure $\fs$ on $X_0$ whose restriction to the boundary is trivial, there is at least one lift $\tilde\fs\in\Spinc(X)$ but also a lift $\tilde\fs+v\in\Spinc(X)$ for each $v\in\text{Im}(\mM\vV)$.

We can also discuss non-symplectic 4-manifolds, for which there is no canonical isomorphism $\Spinc(X)\cong H^2(X)$. We still see that for each spin-c structure $\fs$ on $X_0$ whose restriction to the boundary is trivial, there is at least one lift $\tilde\fs\in\Spinc(X)$ but also a lift $\tilde\fs+v\in\Spinc(X)$ for each $v\in\text{Im}(\mM\vV)$, and the Seiberg--Witten index satisfies $d(\tilde\fs)=d(\tilde\fs+v)$ which is equivalent to $v\cdot\big(v+c_1(\tilde\fs)\big)=0$.
\end{proof}

The symplectic cobordism $(\nN,\omega)$ with relative class $B\in H_2(\nN,UT^* L)\cong H^2(\nN)\cong H^2(L)\cong\ZZ$ defines another ECH cobordism map $ECH_*(\nN,\omega,B)$ using Seiberg--Witten theory \cite{Hutchings:fieldtheory}. We assume $\partial B=\partial A=0$ in order to compose the ECH cobordism maps induced by $X=\nN\circ X_0$, and so $B=0$ because the differential $\partial:H_2(\nN,UT^* L)\to H_1(UT^* L)$ is injective (as seen using the homological long exact sequence). Since $(\nN,\omega)$ is a strong symplectic filling of $(T^3,\xi_0)$, the cobordism map $ECH_*(\nN,\omega,0)$ sends the contact invariants to each other (see \cite{Echeverria:naturality, Hutchings:fieldtheory}) and hence must preserve gradings. So the graded map $ECH_{2k}(\nN,\omega,0):ECH_{2k}(T^3,\xi_0)\to ECH_{2k}(\varnothing,0)$ is trivial for $k\ne0$, and
$$\Phi_\nN:=ECH_0(\nN,\omega,0):ECH_0(T^3,\xi_0,0)\to\ZZ/2$$
sends the nonzero contact invariant $[\varnothing]$ to $1$ (the contact invariant of the empty 3-manifold). In light of the following lemma, the contact invariant is the basis element $\bar\theta\in ECH_0(T^3,\xi_0,0)$.

\begin{lemma}
\label{lem:projection}
The cobordism map\footnote{This cobordism map counts reducible solutions between reducible monopoles. Luckily $b^2_+(\nN)=0$, otherwise \cite{KM:book}*{Proposition 27.2.4} would imply that such (perturbed) moduli spaces are empty.} $\Phi_\nN:\Hto_0(T^3,\fs_0)\to\ZZ$ on monopole Floer homology can be identified with the induced map $H_2(T^3)\to H_2(\nN)$ under the inclusion $\partial\nN\hookrightarrow\nN$, which is the projection $\ZZ\langle x,y,\bar\theta\rangle\to\ZZ\langle\bar\theta\rangle$.
\end{lemma}

\begin{proof}
The unperturbed (blown-down) Seiberg--Witten moduli spaces on $\partial\nN$ and $\nN$ consist solely of reducible configurations, with identifications $M(T^3;\fs_0)\cong\TT^3$ and $M(T^2\times D^2;\fs_\omega)\cong\TT^2$ because all $U(1)$-connections are flat \cite{KM:book}*{Proposition 22.7.1, Lemma 27.2.1}. These identifications are in terms of holonomy, so the restriction map $M(T^2\times D^2;\fs_\omega)\to M(T^3;\fs_0)$ is the natural inclusion $\TT^2\to\TT^2\times\lbrace0\rbrace\subset\TT^3$. Following \cite{KM:book}*{\S37.2}, to compute the Floer groups we slightly perturb the (blown-up) Seiberg--Witten equations using a self-indexing Morse function $h:\TT^3\to\RR$ so that
$$\Hto_0(T^3,\fs_0)\cong\ZZ\cdot\crit_1(h)\cong H_1(\TT^3)$$
We now claim that $\Phi_\nN$ is precisely the map $H_1(\TT^3)\to\ZZ$ given by $\eta\mapsto\eta\cdot[\TT^2\times\lbrace0\rbrace]$. It suffices to prove this claim, because this map is the induced map $H_2(T^3)\to H_2(\nN)$ under the identifications $H_1(\TT^3)\cong H^1(T^3)\cong H_2(T^3)$.

On the chain level, $\Phi_\nN$ counts elements of the perturbed (blown-up) Seiberg--Witten moduli space over $\overline\nN$ which extend the given monopoles corresponding to $\crit_1(h)$. The Morse function perturbation over $T^3$ is extended as a perturbation over $\overline\nN$ using a cutoff function that is supported on the end $[0,\infty)\times T^3$ and equal to 1 on $[1,\infty)\times T^3$, for which the perturbed Seiberg--Witten equations on the end are a gradient-flow. So a monopole $p\in\crit_1(h)$ extends as a gradient-flowline over $[0,\infty)\times T^3$ to an unperturbed monopole on $\partial\nN$, and there is an $S^1$ worth of such flowlines, hence a cycle $\gamma_p\subset\TT^3$ of unperturbed monopoles on $\partial\nN$. A flowline extends over $\nN$ if and only if $[\gamma_p]\cdot[\TT^2\times\lbrace0\rbrace]\ne0$, so the claim is proved. 
\end{proof}

\begin{remark}
Although a description of $\Phi_\nN$ in terms of Seiberg--Witten theory was provided in Lemma~\ref{lem:projection}, here is an indirect description without using Seiberg--Witten theory. Since the nonzero contact invariant is represented by the empty orbit set $\varnothing$ it must generate a summand of homology, $ECH_0(T^3,\xi_0,0)\cong\ZZ/2\langle\Theta_1,\Theta_2,[\varnothing]\rangle$. Thus, $\Phi_\nN:(\ZZ/2)^3\to\ZZ/2$ is some homomorphism determined by
\begin{align*}
(0,0,[\varnothing])&\mapsto1[\varnothing]\\
(\Theta_1,0,0)&\mapsto r_1[\varnothing]\\
(0,\Theta_2,0)&\mapsto r_2[\varnothing]
\end{align*}
So under the change of basis $(\ZZ/2)^3\cong\ZZ/2\langle\Theta_1-r_1[\varnothing],\Theta_2-r_2[\varnothing],[\varnothing]\rangle$ we may identify $\Phi_\nN$ with the projection map onto the 3rd coordinate. Note that $r_1$ and $r_2$ need not be zero: There could exist pseudoholomorphic curves $C$ in the completion $\overline\nN\cong T^* L$ which are asymptotic to $\Theta_i$ with positive energy $\int_C\omega=\int_{\Theta_i}\lambda_0>0$ and simultaneously satisfy $[C]=0\in H_2(\nN,UT^* L)$, such as a cylinder with two positive ends. With regards to Lemma~\ref{lem:projection}, we can use Taubes' isomorphism~\eqref{eqn:Taubes} to pick $\Theta_1$ and $\Theta_2$ in such a way that $r_1=r_2=0$.
\end{remark}

\begin{prop}
Given a relative class $A\in H_2(X_0,UT^* L)$ whose boundary is zero in $H_1(UT^* L)$, the ECH index $I(\tilde A)$ of a lift $\tilde A\in H_2(X)$ is the unique integer for which $g(A,I(\tilde A))=0$.
\end{prop}

\begin{proof}
Uniqueness follows from the description in Theorem~\ref{thm:class} of $\Phi(A,I)$ in terms of the U-map, and the independence of the lift is due to Lemma~\ref{lem:lift}. As both $\Phi_N$ and $\Phi(A,I(\tilde A))$ may be defined by suitable counts of Seiberg--Witten solutions on $\overline\nN$ and $\overline{X_0}$, we look at the index of each Seiberg--Witten moduli space. Using the isomorphisms $\Spinc(X)\cong H_2(X)$ and $\Spinc(X_0)\cong H_2(X_0,UT^* L)$, we have $I(\tilde A)=d(\fs_\omega+\tilde A)$ for the various choices of lifts of $\fs_\omega+A$. Then by the additivity of index (see \cite{KM:book}*{\S24.4}) applied to the composition of cobordisms $X=X_0\circ\nN$,
$$d(\fs_\omega+\tilde A)=\dim\fM(\varnothing,\nN,\fc_\Theta;\fs_\omega)+d(\fs_\omega+\tilde A)$$
for each monopole $\fc_\Theta\in\fM(T^3,\fs_0)$ that contributes to $\Phi(A,I(\tilde A))$. Such a monopole thus satisfies the constraint $\dim\fM(\varnothing,\nN,\fc_\Theta;\fs_\omega)=0$ for any class $A$, including the monopole $\fc_\varnothing$ associated with the class $A=0$. Then
$$|\fc_\Theta|-|\fc_\varnothing|=\gr(\fc_\varnothing,\fc_\Theta)=\dim\fM(\varnothing,\nN,\fc_\varnothing;\fs_\omega)-\dim\fM(\varnothing,\nN,\fc_\Theta;\fs_\omega)=0$$
and hence
$$g(A,I(\tilde A))=|\fc_\Theta|=|\fc_\varnothing|=0$$
The last equality is due to the fact that $\fc_\varnothing$ defines the contact invariant (see Section~\ref{Taubes' isomorphisms}).
\end{proof}

\begin{definition}
\label{defn:Gr}
The ECH invariant of the Lagrangian torus $L\subset(X,\omega)$ in a minimal symplectic 4-manifold, with respect to a relative class $A\in H_2(X_0,UT^* L)$ satisfying $\partial A=0$, is the element
$$Gr_L(A):=\Phi(A,I(\tilde A))\in ECH_0(T^3,\xi_0,0)$$
from Theorem~\ref{thm:class}.
\end{definition}

By the composition law \cite{Hutchings:fieldtheory} the image of $Gr_L(A)$ under $\Phi_\nN$ in $ECH_0(\varnothing,0,0)\cong\ZZ/2$ is the sum of Gromov invariants
$$Gr_X(A):=\sum_{\tilde A}Gr_{X,\omega}(\tilde A)$$
where $\tilde A\in H_2(X)$ indexes over those classes such that $A|_{X_0}=A$ and $A|_\nN=0$. Since $\Phi_\nN$ is the projection map under the identification in Lemma~\ref{lem:projection}, we have
\begin{equation}
\label{eqn:compositionGr}
Gr_L(A)=\big(a,\;b,\;Gr_X(A)\big)\in(\ZZ/2)^3
\end{equation}
We would hope that the remaining pair $(a,b)$ contains new information about $L\subset X$, but we will see momentarily that it simply repackages the Gromov (or Seiberg--Witten) invariants of the surgeries when $\nN$ is attached differently to $X-\nN$ (via the two generating loops of $L$).

%%%%%%%%%%%%%%%%%%%%%%%%%%%%%%%%%%%%%%%%%%%%%%%%%%%%%%%%%%
%%%%%%%%%%%%%%%%%%%%%%%%%%%%%%%%%%%%%%%%%%%%%%%%%%%%%%%%%%
\section{Surgeries}
\label{Surgeries}

In the following statements we use integer coefficients when discussing Seiberg--Witten theory, but we suppress the homology orientations. All homology orientations satisfy a composition law and are derived from a common homology orientation of $X_0$, and we relegate the discussion of them to Appendix~\ref{appendix}.

Since the methodology presented here does not require the symplectic form, we can consider the more general setup: Let $X$ now denote a connected smooth 4-manifold (not necessarily symplectic) with $b^2_+(X)>0$, and let $L$ denote a smoothly embedded 2-torus with trivial self-intersection number. Instead of (relative) homology classes we work with spin-c structures: Given $\tilde\fs\in\Spinc(X)$ we consider the restricted spin-c structures $\fs=\tilde\fs|_{X_0}\in\Spinc(X_0)$ satisfying $\fs|_{\partial X_0}=\fs_0$, and we note that there is a unique spin-c structure (denoted $\fs_\omega$ by abuse of notation) on $\nN$ that extends $\fs_0$ on its boundary. We have already computed the induced monopole Floer cobordism map $\Phi_{(\nN,\,\fs_\omega)}$ in Lemma~\ref{lem:projection}, and similarly the cobordism $X_0$ defines an element of monopole Floer homology
$$SW_L(\fs)\in \Hto_0(T^3,\fs_0)$$
by suitably counting Seiberg--Witten solutions on $\overline{X_0}$ that have index $d(\tilde\fs)$ \cite{KM:book}*{Theorem 3.4.4}. When $X$ is symplectic (and minimal) and $L$ is Lagrangian, this element modulo 2 is the corresponding ECH invariant (as mentioned in the proof of Theorem~\ref{thm:class}).

Given an orientation-preserving diffeomorphism $f\in\Diff_+(T^3)$ and noting that $\partial X_0=-\partial\nN$, let
$$X_f:=X_0\cup_f\nN$$
denote its logarithmic transformation, the torus surgery of $L\subset X$ induced by $f$. By the composition law in \cite{KM:book} we can then compute its Seiberg--Witten invariants\footnote{Strictly speaking, when $b^2_+(X)=1$ the invariants depend on a choice of chamber, but as shown in \cite{KM:book}*{\S27.5} the gluing formula picks out the chamber (dependent on $SW_L(\fs)$ and $\Phi_\nN$).}
\begin{equation}
\label{eqn:composition}
\sum_{\tilde\fs}SW_{X_f}(\tilde\fs)=\Phi_\nN\circ f_*\circ SW_L(\fs)
\end{equation}
where $\tilde\fs\in\Spinc(X_f)$ indexes over the lifts of $(\fs,\fs_\omega)\in\Spinc(X_0)\oplus\Spinc(\nN)$.\footnote{Note that by Proposition~\ref{prop:ECHSWF}, $SW_{X_f}(\fs)=0$ whenever $\fs|_\nN\ne0$.} We denote this sum by
$$SW_{X_f}(\fs):=\sum_{\tilde\fs}SW_{X_f}(\tilde\fs)$$
In other words, $\Phi_\nN\circ f_*:\ZZ^3\to\ZZ$ are various (integral) linear functionals which send the element
\begin{equation}
\label{eqn:compositionSW}
SW_L(\fs)=\big(a,\;b,\;SW_X(\fs)\big)\in\ZZ^3
\end{equation}
to the various Seiberg--Witten invariants of torus surgeries. We will now pin down this pair $(a,b)$ using enough choices of $f$, and when $(X,L)$ is our symplectic-Lagrangian pair we will use only Luttinger surgeries (of symplectic manifolds) instead of all logarithmic transformations (of smooth manifolds). In this regard, a Luttinger surgery $X\mapsto X_f$ preserves symplecticity and minimality.

\begin{theorem}
\label{thm:main}
Given an integer $p\in\ZZ$ and a loop $\gamma=rx+sy\in H_1(T^2)$ with $r,s\in\ZZ$, let $f_{p,r,s}:T^2\times S^1\to T^2\times S^1$ be an orientation-preserving diffeomorphism determined by $f_*\bar\theta=p\bar\theta-\gamma$. Then for $\fs\in\Spinc(X_0)$ such that $\fs|_{T^3}=\fs_0$,
$$SW_L(\fs)=\left(SW_{X_{f_{0,1,0}}}(\fs),\;SW_{X_{f_{0,0,1}}}(\fs),\;SW_X(\fs)\right)\in\ZZ^3\cong\Hto_0(T^3,\fs_0)$$
where in the $b^2_+(X)=1$ case the invariants are computed in corresponding chambers. If $X$ is symplectic and $L$ is Lagrangian, then
$$Gr_L(A)=\left(Gr_{X_{f_{1,1,0}}}(A)-Gr_X(A),\;Gr_{X_{f_{1,0,1}}}(A)-Gr_X(A),\;Gr_X(A)\right)\in(\ZZ/2)^3\cong ECH_0(T^3,\xi_0,0)$$
where $A\in H_2(X_0,\partial X_0)$ corresponds to $\fs\in\Spinc(X_0)$ via the symplectic form.
\end{theorem}

\begin{proof}
Given $f\in\Diff_+(T^2\times S^1)$, the induced map on 2nd homology is represented by any matrix of the form $\left( \begin{smallmatrix}
*&*&r\\ *&*&s \\ *&*&p
\end{smallmatrix} \right)\in SL_3(\ZZ)$.

Since $\Phi_\nN$ is projection onto the 3rd coordinate by Lemma~\ref{lem:projection}, we look to compute the pair $(a,b)$ by permuting the basis of $\ZZ^3\cong H_2(T^3)$. The permutation matrices are given by general torus surgeries: $(p,r,s)=(1,0,0)$ and $(p,r,s)=(0,1,0)$ and $(p,r,s)=(0,0,1)$. Thus,~\eqref{eqn:composition} and~\eqref{eqn:compositionSW} together imply $a=SW_{X_{f_{0,1,0}}}(\fs)$ and $b=SW_{X_{f_{0,0,1}}}(\fs)$, where in the $b^2_+(X)=1$ case the invariants are computed in the chambers determined by the gluing formula \cite{KM:book}*{\S27.5}.

Now suppose that $p=1$. Then $f$ is a contactomorphism of $(T^3,\xi_0)$ if and only if $f_*$ acts trivially on $H_1(T^2)$ \cite{EP:Luttinger}*{Theorem 1.3.A}; these are the Luttinger surgeries \cite{Luttinger, ADK:Luttinger}. In that case, $f_*=\left( \begin{smallmatrix}
1&0&r\\0&1&s\\0&0&1
\end{smallmatrix} \right)$ and so~\eqref{eqn:composition} and~\eqref{eqn:compositionGr} together imply the following ``Gromov equations''
$$Gr_{X_{f_{1,r,s}}}(A)\equiv Gr_X(A)+ra+sb\mod2$$
We compute $a$ and $b$ by solving the system of two ``Gromov equations'' defined by $(r,s)=(1,0)$ and $(r,s)=(0,1)$.
\end{proof}

From the proof of this theorem we immediately obtain the following product formulas along $T^3$ of the Gromov and Seiberg--Witten invariants. The Seiberg--Witten formula appearing below is the main result of \cite{MMS:product}, as alluded to in the ``notes and references'' at the end of \cite{KM:book}*{Chapter IX}.\footnote{We are not using twisted local coefficients, so \cite{KM:book}*{\S38.2} does not apply here. Such coefficient systems may instead be used to recover the main result of \cite{Taubes:tori}.}

\begin{cor}
\label{cor:main}
In the notation of Theorem~\ref{thm:main},
$$Gr_{X_{f_{1,r,s}}}(A)\equiv_{(2)}r\cdot Gr_{X_{f_{1,1,0}}}(A)+s\cdot Gr_{X_{f_{1,0,1}}}(A)+(1-r-s)\cdot Gr_X(A)$$
and
$$SW_{X_{f_{p,r,s}}}(\fs)=p\cdot SW_X(\fs)+r\cdot SW_{X_{f_{0,1,0}}}(\fs)+s\cdot SW_{X_{f_{0,0,1}}}(\fs)$$
where in the $b^2_+(X)=1$ case the Seiberg--Witten invariants are computed in corresponding chambers.
\end{cor}

\begin{remark}
Generally, given a link of $N$ tori we get elements in $\bigotimes_{k=1}^NECH_0(T^3,\xi_0,0)\cong(\ZZ/2)^{3^N}$ via a straightforward generalization of Theorem~\ref{thm:class}, and these tuples of integers repackage the various Seiberg--Witten invariants obtained by performing torus surgeries on the generating loops of each torus.
\end{remark}

\begin{remark}
The gauge-theoretic invariant of $(X,L)$ specified in \cite{FS:Lagrangian}*{Proposition 2.1} is precisely the collection of the invariants $\sum_A\big(\sum_{\fs\;|\;c_1(\fs)=A}SW_{X_{f_{p,r,s}}}(\fs)\big) A\in \ZZ H_2(X_0,\partial X_0)$ as each $p,r,s$ ranges over the integers, and thus is recovered by the collection of invariants $SW_L(\fs)$ as $\fs$ ranges over relative spin-c structures on $X_0$. In particular, for $X$ symplectic (and minimal) and $L$ Lagrangian, this collection of (mod 2) Seiberg--Witten invariants may be recovered by the collection of ECH invariants $Gr_L(A)$ thanks to the gluing formula~\eqref{eqn:composition}.
\end{remark}

%%%%%%%%%%%%%%%%%%%%%%%%%%%%%%%%%%%%%%%%%%%%%%%%%%%%
%%%%%%%%%%%%%%%%%%%%%%%%%%%%%%%%%%%%%%%%%%%%%%%%%%%%
\appendix\section{Appendix: homology orientations}
\label{appendix}

As explained in \cite{KM:book}*{\S 20, \S28.4}, to define the monopole Floer groups and cobordism maps we must first choose a \textit{(cobordism) homology orientation} of $X:Y_+\to Y_-$. This is an orientation of
$$\operatorname{det}^+(X):=\det H^1(X;\RR)\otimes\det I^+(X;\RR)\otimes\det H^1(Y_+;\RR)$$
where $I^+(X;\RR)$ is defined as follows (see also \cite{KM:book}*{\S3.4}): The relative cap-product pairing
$$H^2(X,\partial X;\RR)\times H^2(X;\RR)\to H^4(X,\partial X;\RR)\cong\RR$$
induces a nondegenerate quadratic form on the kernel of the restriction map $H^2(X;\RR)\to H^2(\partial X;\RR)$, and $I^+(X;\RR)\subset H^2(X;\RR)$ is a maximal nonnegative subspace for this quadratic form. The set of homology orientations is denoted by $\Lambda(X)$.

\begin{remark}
In the case that $Y_\pm=\varnothing$, we recover the notion of homology orientation of a closed 4-manifold in Section~\ref{Closed 4-manifolds}. When this closed 4-manifold is equipped with a symplectic form there is a canonical homology orientation \cite{Taubes:Gr=SW}*{\S 1.c}.
\end{remark}

There is also a composition law for (cobordism) homology orientations \cite{KM:book}*{\S3.4, \S26.1}. The relevant 4-manifold $X$ in this paper can be viewed as the composition of cobordisms
$$\varnothing\overset{\nN}{\xleftarrow{\hspace*{.75cm}}}T^3\overset{X_0}{\xleftarrow{\hspace*{.75cm}}}\varnothing$$
and then there is a specification
$$\Lambda(X)=\Lambda(\nN)\otimes_{\ZZ/2}\Lambda(X_0)$$
so that a choice of homology orientation for any two objects in $\lbrace X,\nN,X_0\rbrace$ determines a homology orientation of the third object.

A \textit{homology orientation} of $T^3$ is an orientation of the vector space $H^1(T^3;\RR$), and this was fixed at the beginning of Section~\ref{Floer homologies of 3-tori}. It is used to specify the duality isomorphism between monopole Floer (co)homology groups upon orientation-reversal of $T^3$ (see Section~\ref{Gradings}).

With that said, the main property of homology orientations implicitly used in this paper is the following.

\begin{prop}
Choose a basis of $H^1(T^2;\RR)$ and $H^1(T^3;\RR)$. Then a homology orientation of $X$ determines that of its torus surgeries $X_f$ for $f\in\Diff_+(T^3)$.
\end{prop}

\begin{proof}
By the composition law it suffices to show that the chosen basis pins down the homology orientation of $\nN$. But this is immediate, because $I^+\subset\text{Ker}\left(H^2(\nN;\RR)\to H^2(\partial\nN;\RR)\right)=0$ and hence $\Lambda(\nN)\cong \det H^1(T^3;\RR)\otimes\det H^1(T^2;\RR)$.
\end{proof}

%%%%%%%%%%%%%%%%%%%%%%%%%%%%%%%%%%%%%%%%%%%%%%%%%%%%
%%%%%%%%%%%%%%%%%%%%%%%%%%%%%%%%%%%%%%%%%%%%%%%%%%%%

\end{document}